\documentclass{article}

\pdfoutput=1

\usepackage{graphicx} % Required for inserting images
\usepackage[utf8]{inputenc}

\usepackage[square,numbers]{natbib}
\usepackage{csquotes}

\usepackage{graphicx}     % za slike

\usepackage{amsmath}
\usepackage{amssymb}
\usepackage{amsthm}
\usepackage{bbm}
\usepackage[nobottomtitles*,compact]{titlesec}

\usepackage{caption}
\usepackage{subcaption}

\usepackage{xcolor}

\usepackage{pgfplots}
\pgfplotsset{compat=1.15}
\usepackage{mathrsfs}
\usetikzlibrary{arrows}

\definecolor{ffudud}{rgb}{1.,0.30196078431372547,0.30196078431372547}
\definecolor{ududff}{rgb}{0.30196078431372547,0.30196078431372547,1.}
\definecolor{ttttff}{rgb}{0.2,0.2,1.}

\usepgfplotslibrary{fillbetween}

\usepackage{geometry}
\usepackage{placeins}

\usepackage{subcaption}

\usepackage{quiver}
\usetikzlibrary{babel}
\usetikzlibrary{nfold}

\newtheorem{thm}{Theorem}[section]
\newtheorem{prop}[thm]{Proposition}
\newtheorem{lemma}[thm]{Lemma}
\newtheorem{cor}[thm]{Corollary}

\theoremstyle{definition}
\newtheorem{defi}[thm]{Definition}
\newtheorem{rem}[thm]{Remark}
\newtheorem*{primer*}{Example}

\newenvironment{ex}[1][]{
  \begin{primer*}[#1]\pushQED{\qed}
}{
  \popQED\end{primer*}
}

\usepackage{tikzit}
\tikzstyle{dot}=[fill={rgb,255: red,191; green,191; blue,191}, draw=black, shape=circle]

\tikzstyle{arrow}=[->, ultra thick, line width=2pt]
\tikzstyle{no_arrow}=[-, line width=2pt]
\tikzstyle{new edge style 0}=[-]
\tikzstyle{dotted}=[-, dashed, line width=2pt]

\usepackage{array}% http://ctan.org/pkg/array

\usepackage{nicematrix}

\DeclareMathOperator{\Mat}{\mathit{Mat}}

\DeclareMathOperator{\im}{im}

\DeclareMathOperator{\UT}{UT}

\DeclareMathOperator{\rad}{rad}
\DeclareMathOperator{\soc}{soc}

\DeclareMathOperator{\urep}{\underline{\smash{\rep}}}
\DeclareMathOperator{\rep}{rep}

\newcommand\mbb[1]{\mathbb{#1}}

\DeclareMathOperator{\rank}{rank}

\DeclareMathOperator{\GL}{GL}

\DeclareMathOperator{\Mod}{\underline {mod}}

\DeclareMathOperator{\raddim}{\underline{rdim}}
\DeclareMathOperator{\socdim}{\underline{sdim}}
\DeclareMathOperator{\layer}{\mbb{S}}
\DeclareMathOperator{\rl}{rl}
\DeclareMathOperator{\socl}{sl}

\newcommand{\Addresses}{{% additional braces for segregating \footnotesize
  \bigskip
  \footnotesize
  M.~Čmrlec, \textsc{
  (1) KU Leuven, Celestijnenlaan 200B, B-3001 Leuven, Belgium; (2) BCAM,
Basque Center for Applied Mathematics, Mazarredo 14, 48009 Bilbao, Basque Country,
Spain.}
  \par\nopagebreak
  \textit{E-mail address}: \texttt{marko.cmrlec@kuleuven.be}
}}

\def\keywords{\xdef\@thefnmark{}\@footnotetext}

\title{Irreducible components of varieties of representations of a class of finite dimensional algebras}
\author{Marko Čmrlec}

\begin{document}

\maketitle

\begin{abstract}
\noindent Let $\mathcal{A}=k\langle x_1,\dots,x_n\rangle/((x_1,\dots,x_n)^3+(S)$, where $S=\sum_{i,j=1}^n a_{ij}x_ix_j$ with $\det(a_{ij})\neq 0$ be an algebra. We determine the irreducible components of the variety of representations $\rep_d(\mathcal{A})$. We do this by showing that the subvarieties with fixed radical or socle layering are irreducible. We then compare the coverings we get using radical and socle layerings to determine the components.
\end{abstract}

{
\setlength{\parskip}{0pt}
\tableofcontents
}
\section{Introduction}

The main goal of representation theory is to understand all modules of a given algebra and their homomorphisms. This problem is equivalent to classifying the indecomposable modules and determining their module-theoretic properties. However, for most algebras, this seems to be a hopeless task. Because of that, we would like to simplify the problem and focus on understanding just the generic modules. There exist natural varieties that parametrise representations of a fixed algebra. These varieties further admit an action of a reductive group, such that the orbits are isomorphism classes of representations. This problem has first been studied in~\cite{Kac1980, KAC1982141} 
for hereditary algebras. In this case, the varieties are irreducible so it makes sense to speak of generic points and generic properties of such varieties. In general, however, the varieties have many irreducible components. So the task in this case is to determine the irreducible components and generic properties of representations in each component. For a thorough survey of this problem check~\cite{genRepTheory}.

The problem has most successfully been studied in the case of hereditary algebras, where the dimensions of generic indecomposable summands are understood~\cite{Schofield1992GeneralRO}.

The other successful case is for tame non-hereditary algebras. The problem has been resolved for several classes of such algebras using the known classifications of indecomposable representations. For example, this has been done for Gelfand-Ponomarev algebras $k[x,y]/(x^r,y^s,xy)$ by~\cite{schroer_tame}, acyclic gentle string algebras by~\cite{caroll_tame}, and preprojective algebras $P(Q)$ where $Q$ is quiver of Dynkin type by~\cite{ChristofGeiss2003}.

For wild non-hereditary algebras, the problem is more difficult. In this case, it has only been solved for truncated path algebras~\cite{ ClosuresVarietiesReps, IrrCompLocal}. These proofs use the fact, that stratification by powers of radicals in this case induces the covering of the variety of representations by affine spaces. We extend this method by adding a relation. This means that the covering varieties are not affine spaces. However, for appropriate relations, they remain irreducible. 

A related problem is classification of algebras with irreducible representation varieties. This has been done for quotients of path algebras of quivers with at most two vertices by~\cite{BOBINSKI2019624, BOBINSKI2024107707}.

In the following, we determine the components for varieties of representations of a class of finite dimensional algebras, that contains all algebras of the form $\mathcal{A}=k\langle x_1,\dots,x_n\rangle/((x_1,\dots,x_n)^3+(\sum_i x_i^2))$. For these, we show that the irreducible components are given by the closures of varieties $\rep_{\underline{d}} (\mathcal{A})$ of representations with given radical layering.

\begin{thm}\label{thm_components}
    Let $k$ be an algebraically closed field with $\operatorname{char}(k)\neq 2$, and $\mathcal{A}=k\langle x_1,\dots,x_n\rangle/I$ the algebra of the algebra of noncommutative polynomials, quotiented by the ideal

    \[I=(x_1,\dots,x_n)^3+(S)\]

    where $S=\sum_{i,j=1}^n a_{ij}x_ix_j$ with $\det(a_{ij})\neq 0$. Let $d\in\mbb{N}$.

    Then $\rep_d \mathcal{A}$ admits a decomposition into irreducible components

    \[\rep_d \mathcal{A}=\bigcup_{\underline{d}\in \Lambda}\overline{\rep_{\underline{d}}(\mathcal{A})}\]

    Where $\Lambda$ is the set of dimension vectors $\underline{d}=(d_0,d_1,d_2)$, which satisfy $d_0+d_1+d_2=d$ and

    \begin{align*}
        d_1&\leq n\cdot d_0 & d_1&\leq n\cdot d_2  \\
        d_2&\leq n\cdot d_1-d_0 &  d_0&\leq n\cdot d_1-d_2, 
    \end{align*}

    except in the case $d\equiv 1\pmod{n^2+n}$, when $\Lambda$ additionally contains the dimension vectors

    \[(a,n(a-1),(n^2-1)(a-1)) \qquad \text{and} \qquad ((n^2-1)(a-1),n(a-1)+1,a-1)\]

    for $a\in\mbb{N}$ such that $d=(n^2+n)(a-1)+1$.
\end{thm}

\begin{rem}
    Due to the duality between radicals and socles, the above theorem also holds if we replace $\rep_{\underline{d}}( \mathcal{A})$ with varieties $\rep^*_{\underline{d}}( \mathcal{A})$ of matrices with fixed socle layering.
\end{rem}

Additionally, we determine the generic dimensions of socle layering for these varieties.

\begin{thm}\label{thm_gen_rad}
    Let algebra $\mathcal{A}$ be as above and $\underline{d}=(d_0,d_1,d_2)$ be a dimension vector for radical layering, such that $\rep_{\underline{d}}(\mathcal{A})$ is an irreducible component of $\rep_d(\mathcal{A})$. Then the generic dimension of socle layering for $\rep(\underline{d},\mathcal{A})$ equals;
    
    \begin{align*}
        \socdim \rep (a,n(a-1),(n^2-1)(a-1))&= ((n^2-1)(a-1),n(a-1)+1,a-1),\\
        \socdim \rep ((n^2-1)(a-1),n(a-1)+1,a-1)&= (a,n(a-1),(n^2-1)(a-1)),\\
        \socdim \rep (d_0,d_1,d_2)&=(d_2,d_1,d_0) & \text{otherwise.}
    \end{align*}
\end{thm}

In section~\ref{defs} we give the exact definitions of varieties $\rep_d \mathcal{A}$, $\rep({\underline{d}}, \mathcal{A})$, $\rep_{\underline{d}}(\mathcal{A})$, for arbitrary path algebra of a quiver as well as some semicontinuous maps we will use. 

In section~\ref{irr} we show that the stratification of representations with their radicals and socles induces a structure of a kernel of a morphism of vector bundles on varieties $\rep(\underline{d},\mathcal{A})$ and compute their dimensions of fibres. This gives us a condition for when these varieties are vector bundles themself. We use this to show that for the considered algebras, the varieties $\rep(\underline{d}, \mathcal{A})$ and $\rep(\underline{d})$ are irreducible. A similar idea has been used before by~\cite{BOBINSKI2019624} for a different stratification.

In section~\ref{possible_layerings} we determine which $\rep(\underline{d},\mathcal{A})$ are nonempty.

In section~\ref{incl} we determine which of the above varieties are maximal for inclusion, and use this to determine which are components. We use the same semicontinuous maps as in~\cite{IrrCompLocal} to separate the components. Their method to show inclusions, however, does not generalise. For this, instead of directly comparing the closures varieties with fixed radicals, we compute their generic socle layerings and use them to determine possible components.

\section{Notation and definitions}~\label{defs}
\subsection*{Representation theory}
In this text, $\mathcal{A}$ is a finite-dimensional unital associative (but not necessarily commutative) algebra over an algebraically closed field $k$ of characteristic  $\operatorname{char} k\neq 2$. A module $M$ is a finitely generated left $\mathcal{A}$-module. We say it has dimension $\dim M=d$ if it has dimension $d$ as a $k$-vector space. 

For an $\mathcal{A}$-module $M$ its Jacobson radical $\rad M$ is the intersection of all maximal left submodules of $M$. It is well known that for modules over Artin algebras, $\rad M=\rad \mathcal{A} \cdot M$. Similarly, we define $i$-th radical $\rad^iM=\rad^i\mathcal{A}\cdot M$. Again because $\mathcal{A}$ is Artin the series of radicals is finite and for some $m$, $\rad^m\mathcal{A}=0$. Hence we get a filtration

\[M\supseteq \rad M \supseteq \dots \supseteq\rad^m M=0\]

If $l\leq m$ is the least integer such that $\rad^l M=0$ we say that $l=\rl M$ is the radical length of $M$. The sequence of factors $(\rad^i M/\rad^{i+1} M)_i$ is called the radical layering of $M$. This is a sequence of semisimple $M$ modules. Hence it is determined by the dimensions of its isotypical components.

Dually, the socle $\soc M$ is the largest semisimple submodule of $M$. For modules over an Artin algebra, there is an alternative characterisation 

\[\soc M=\{m\in M\; | \; (\rad \mathcal{A})\cdot m=0\}\]

so we can define $i$-th socle as 

\[\soc^i M=\{m\in M\; | \; (\rad^i \mathcal{A})\cdot m=0\}\]

Similarly to radical, we get a filtration 

\[0\subseteq \soc M \subseteq \dots \subseteq\soc^m M=M\]

We call the sequence of factors $(\soc^{i+1}M/\soc^i M)_i$ the socle layering of $M$. This is also a sequence of semisimple modules.
We also define socle length $\socl M$ as the minimal $l$, such that $\soc^l M=M$. For every Artin algebra $\mathcal{A}$ and $\mathcal{A}$-module $M$, $\socl M=\rl M$. 

Since socle is dual to radical, the results we prove for radicals have dual versions for socles. From now on we will focus on the results for radicals and note when similar construction also works for socles.

By Morita equivalence, it is enough to consider the path algebras of quivers. A quiver $\Gamma=(\Gamma_0,\Gamma_1,s,e)$ is a directed multigraph that may have loops. The set of vertices of $\Gamma$ is $\Gamma_0$ and the set of its edges is $\Gamma_1$. To each edge $f\in \Gamma_1$ we associate its start $s(f)$ and end $e(f)$. A path is $p$ in $\Gamma$ starting at vertex $a=s(p)$ and ending at vertex $b=e(p)$ is a (possibly length $0$) sequence of edges $f_1,\dots, f_n$ with $s(f_1)=a$, $e(f_i)=s(f_{i+1})$ and $e(f_n)=b$. We write $e_a$ for the length $0$ path at $a\in \Gamma_0$.

Path algebra $k\Gamma$ of $\Gamma$ is the $k$ algebra generated by paths in $\Gamma$, where product $p\cdot q$ is the concatenation $p$ after $q$ if $e(q)=s(p)$ and $0$ otherwise. An element $c=\sum c_ip_i\in k\Gamma$ is homogeneous if there exist $s(c),e(c)$ such that $s(c)=s(p_i)$ and $p(c)=e(p_i)$ for all $i$ with $c_i\neq 0$. We define $k\Gamma^{+n}$ as the (two-sided) ideal of all paths of length at least $n$. A (two sided) ideal $I\subseteq k\Gamma$ is admissible iff there exists some $n\in\mbb{N}$ such that $k\Gamma^{+n}\subseteq I\subseteq k\Gamma^{+2}$.

A representation $(\varphi,V)$ of a path algebra $k\Gamma$ associates to each vertex $a$ in $\Gamma$ a vector space $V(a)$ and each edge $f$ a map $\varphi(f):V(s(f))\rightarrow V(e(f))$. We can extend $\varphi$ by compositions and linearity to homogeneous elements $ k\Gamma$. $(\varphi,V)$ is a representation of the path algebra $k\Gamma/I$ of quiver $\Gamma$ with relations in $I$. if $\varphi(c)=0$ for all homogeneous elements $c\in I$. For a representation $(\varphi,V)$ we define the dimension vector $d=(d(a))_{a\in \Gamma_0}=(\dim V(a))_{a\in \Gamma_0}$.

A morphism $\Phi\in \hom((\varphi,V),(\vartheta,W))$ is a set of linear maps $\Phi(a):V(a)\rightarrow W(a)$ that commute with the actions $\varphi(f)$ and $\vartheta(f)$ for all edges $f\in k\Gamma$. We write $\urep(k\Gamma/I)$ for the category of representations of $k\Gamma/I$ with morphisms as above. There exists an equivalence between $\urep(k\Gamma/I)$ and $k\Gamma/I-\Mod$. So for path algebras with relations, we may study the category $\urep(k\Gamma/I)$ instead of $k\Gamma/I-\Mod$.

\subsection*{Varieties of representations}

Let $\mathcal{A}=k\Gamma/I$ be a path algebra with relations and fix an ordering $f_1,\dots f_n$ of edges in $\Gamma_1$.

Then the set of representations of $\mathcal{A}$ of dimension $d$,

\[\rep_d \mathcal{A}\subset \prod_{f\in \Gamma_1} \Mat(k;e(f)\times s(f))\cong \mbb{A}[x_{ij}^f]\]

with the reduced induced scheme structure, is a subvariety\footnote{For us variety means a separated reduced scheme of finite type over an algebraically closed field $k$. We treat closed subsets of a variety as varieties with the reduced induced subscheme structure.}. Indeed if we let $x_{ij}^f$ be the coefficient in position $(i,j)$ of the matrix $\varphi(f)$, then $\rep_d \mathcal{A}=V(\Tilde{I})$, for

\[\Tilde{I}=\left(r((x_{ij}^{f_1}),\dots,(x_{ij}^{f_n}))_{ab} \; |\; r\in I, \text{ homogeneous } \: 1\leq a,b\leq d\right)\]

where $r(M_1,\dots,M_n)_{ab}$ is the polynomial, that computes the $(a,b)$ coefficient of the matrix $r(M_1,\dots,M_n)$.

The variety $\rep_d \mathcal{A}$ is stable for the natural action of the group $\prod_{a\in \Gamma_0}\GL_{d(a)}(k)$ by conjugation and its orbits are the isomorphism classes of representations of $\mathcal{A}$. Since the group  $\prod_{a\in \Gamma_0}\GL_{d(a)}(k)$ is irreducible, all the irreducible components of $\rep_d \mathcal{A}$ must also be stable under conjugation. 

By our assumption $I\subseteq k\Gamma^{+2}$ the simple submodules of $\mathcal{A}$ are in bijection with the vertices $\Gamma_0$. Because of that, the radical layering of a module $M$ can be identified with its dimension vector $(d_i)_i$, where $d_i(a)=\dim \left(\rad^i V/\rad^{i+1} V\right)(a)$. For a dimension vector $(d_i)_i$ we consider the subvariety of representations for which the radical layering has dimension vector $(d_i)_i$.

\begin{defi}
    Let $k=\rl(k\Gamma/I)$ and let $\underline{d}=(d_i)_{i\leq m}$ be a dimension vector with $\sum_i d_i=d$, and let $e_1,\dots e_{d(a)}$ be the standard basis for $k^{d(a)}$ for all $a$. Then the variety of representations with radical layering $\underline{d}$ is

    \[\rep(\underline{d},\mathcal{A})=\left\{V\in \rep_d \mathcal{A} \; | \; \rad^i(V)(a)=\langle e_1,\dots e_{D_i(a)}\rangle     \right\}\]

    where $D_i(a)=\sum_{j\geq i} d_i(a)$.
\end{defi}

Since $\rad \mathcal{A} \cdot \rad^i(M)= \rad^{i+1} M$, the representations in $\rep(\underline{d},\mathcal{A})$ have a block upper triangular form

\begin{equation}\label{blok}
    \varphi(f_i) =\begin{array}{c @{} ccccc @{} l}
            &  d_{m-1}(s(f))& d_{m-2}(s(f))&\dots& d_0(s(f)) & \\       
        \left( 
        \vphantom{\begin{array}{@{} c @{}}  0  \\ 0 \\ \vdots \\ 0 \end{array}} \right. &
        \begin{array}{@{} c @{}}  0  \\ 0 \\ \vdots \\ 0 \end{array} &
        \begin{array}{@{} c @{}} \varphi_M(f_i)_{m-1,m-2} \\ 0 \\ \vdots \\ 0 \end{array} & 
        \begin{array}{@{} c @{}} \dots \\ \ddots \\ \ddots \\ \dots \\ \end{array} &
        \left. 
        \begin{array}{@{} c @{}} \varphi_M(f_i)_{m-1,0} \\  \varphi_M(f_i)_{m-2,0} \\ \vdots \\ 0 \end{array}
        \right)
         & \begin{array}{@{} c @{}}
            d_{m-1}(e(f))  \\
            d_{m-2}(e(f)) \\
            \vdots \\
            d_0(e(f))
         \end{array}
    \end{array}
\end{equation}

$\rep(\underline{d},\mathcal{A})$ is a (quasi-affine) variety. To see this consider the following lemma:

\begin{lemma}\label{lema_odp}
    Let $(\varphi, V)$ be a representation, such that the matrix $\varphi(f)$ has a block upper triangular form~\ref{blok} for all $f\in \Gamma_1$. Then $\underline{d}=\raddim V$ iff for all $a\in \Gamma_0$ in $j\leq m-2$ the following holds

    \begin{equation}\label{pogoj'}
        d_{j+1}(a)=\dim \sum_{e(f)=a} \im \varphi(f_i)_{j+1,j}
    \end{equation}
\end{lemma}

For each $j\leq $ and $a$, the condition~\ref{pogoj'} is open, since the representation satisfies this condition iff the matrix

\[(\varphi(f_1)_{j+1,j},\dots,\varphi_M(f_l)_{j+1,j})\]

where $f_i$ ranges over all $f\in \Gamma_1$ with $e(f)=a$
has full rank i.e. if there exists some maximal minor of this matrix, that does not vanish, for all $a\in \Gamma_0$. Hence this is an open subset of the variety of block upper triangular representations of $\mathcal{A}$, with given block sizes.

\begin{proof}[Proof of lemma~\ref{lema_odp}]
    First assume that the condition holds for all $i,j$. The chosen blocks define a splitting $V(a)=\bigoplus V_i(a)$. Let $W_i(a)=\bigoplus_{j\geq i}V_{j}(a)$. These define a subrepresentation, $W$. We wish to prove $\rad^i V=W_i$. The containment $\rad^i  V\subseteq W_i$ is clear since the matrices are upper triangular.

    We show the other containment by induction backwards.
    For $i=m-1$ the proposition is clear. So suppose it holds for $i+1$. Let $v\in W_{i}(a)$. Then there exist $s\in V_i(a)$ and $w\in W_{i+1}(a)$ such that $v=s+w$. The condition~(\ref{pogoj'}) then guarantees, that there exists $w'\in W_{i+1}(a)$, such that $s+w'\in \rad \mathcal{A}\cdot W_{i-1}(a)$. Indeed, $s$ can be written as a sum $s=\sum_f \varphi(f)_{i+1,i}(s_f)$, where $s_f\in V_{i-1}(s(f_j))$. So if we let $s'=\sum_f \varphi(f) s_f$, then $s'\in \rad \mathcal{A}\cdot W_{i-1}(a)$. Additionally by definition $w'=s-s'\in W_{i+1}(a)$. Hence $w+w'\in\rad \mathcal{A} \cdot W_{i}(a)$ and consequently

    \[v=s+w=s'+w'+w\in \rad \mathcal{A}\cdot W_{i-1}(a)\]

    The converse is also clear. Indeed if $i$ is the least number such that the condition~\ref{pogoj'} is not satisfied then by above $\rad^{i-1} \mathcal{A}\cdot V\subseteq W_{i-1}=W_i\oplus M_{i-1}$. Hence 
    
    \begin{align*}
        \rad^i \mathcal{A}\cdot M&\subseteq \rad \mathcal{A} \cdot W_i\oplus \rad \mathcal{A} \cdot M_{i-1}\\
        &\subseteq W_{i+1}\oplus \rad \mathcal{A}\cdot M_{i-1}
    \end{align*}
    
    which cannot contain whole $W_i(a)$, since by assumption $\im\varphi(f)_{i,i-1}$ do not cover $V_i(a)$. 
\end{proof}

Similarly, we can define the variety with given socle layering;

\begin{defi}
    Let $k=\socl(k\Gamma/I)$ and let $\underline{d}=(d'_i)_{i\leq m}$ be a dimension vector with $\sum_i d_i=d$ and let $e_1,\dots e_d$ be the standard basis for $k^d$. Then the variety of representations with socle layering $(d_i)_i$ is

    \[\rep^*(\underline{d},\mathcal{A})=\left\{V\in \rep_d \mathcal{A} \; | \; \soc^i(V)(a)=\langle e_1,\dots e_{D_i'(a)}\rangle     \right\}\]

    where $D'_i(a)=\sum_{j\leq i} d'_i(a)$.
\end{defi}

so matrices in $\rep^*(\underline{d},\mathcal{A})$ have a block upper triangular form

\begin{equation}\label{blok'}
    \varphi(f_i) =\begin{array}{c @{} cccc @{} l @{} l}
            &  d'_{0}(s(f))& d'_{1}(s(f))&\dots& d'_{m-1}(s(f)) & & \\       
        \left( 
        \vphantom{\begin{array}{@{} c @{}}  0  \\ 0 \\ \vdots \\ 0 \end{array}} \right. &
        \begin{array}{@{} c @{}}  0  \\ 0 \\ \vdots \\ 0 \end{array} &
        \begin{array}{@{} c @{}} * \\ 0 \\ \vdots \\ 0 \end{array} & 
        \begin{array}{@{} c @{}} \dots \\ \ddots \\ \ddots \\ \dots \\ \end{array} &
        \begin{array}{@{} c @{}} * \\  * \\ \vdots \\ 0 \end{array} &
        \left. \vphantom{ 
        \begin{array}{@{} l @{}} * \\  * \\ \vdots \\ 0 \end{array}}
        \right) & 
        \begin{array}{@{} c @{}}
            d'_{0}(e(f))  \\
            d'_{1}(e(f)) \\
            \vdots \\
            d'_{m-1}(e(f))
         \end{array}
    \end{array}
\end{equation}

This is again a quasi-affine variety.

We have

\[\rep_d \mathcal{A}=\coprod_{\underline{d}\in \Lambda} \GL_d(k)\cdot\rep(\underline{d},\mathcal{A})=\coprod_{\underline{d}'\in \Lambda'} \GL_d(k)\cdot\rep^*(\underline{d}',\mathcal{A})\]

where $\Lambda$ runs over all possible dimensions of radical layerings and $\Lambda'$ runs over all possible dimensions of socle layerings of $d$-dimensional $\mathcal{A}$-modules. Hence all irreducible components of $\rep_d \mathcal{A}$ must at the same time be closures of irreducible components of $\GL_d(k)\cdot \rep(\underline{d}, \mathcal{A})$ and $\GL_d(k)\cdot \rep^*(\underline{d}, \mathcal{A})$.

We will use $\rep_{\underline{d}}(\mathcal{A})$ to denote the variety $\GL_d(k)\cdot \rep(\underline{d},\mathcal{A})$, and $\rep_{\underline{d}}^*(\mathcal{A})$ for variety $\GL_d(k)\cdot \rep(\underline{d},\mathcal{A})$.

\subsection*{Semicontinuity}

In this subsection, we recall the basic properties of semicontinuous maps. 

Let $X$ be a topological space and $(\mathcal{P},\leq)$ a poset. For $p\in\mathcal{P}$, we define $[p,\infty)=\{x\in \mathcal{P} \; | \; p\leq x \}$. Similarly we define $(p,\infty),(-\infty,p)$ and $(-\infty,p]$. 

We say a map $f:X\rightarrow \mathcal{P}$ is upper semicontinuous iff preimage $f^{-1}([p,\infty)$ is closed in $X$ for all $p\in\mathcal{P}$. Similarly we say $f$ is lower semicontinuous if the preimage $f^{-1}((-\infty,p])$ is closed for all $p\in\mathcal{P}$.

For us, the most important semicontinuous maps will be the following.

\begin{defi}
    Let $\mathcal{A}$ be a finite dimensional algebra with $\rl \mathcal{A}=k$
    Consider the partially ordered set $\mathcal{P}=\left(\mbb{N}^{|{\Gamma_0}|}\right)^{m}$ with order given by $(x_i)_i\prec (y_i)_i$ iff for all $j\leq m$

    \[\sum_{i\leq j} x_i\leq \sum_{i\leq j} y_i \]

    where $x_i\leq y_i$ iff $x_i(a)\leq y_i(a)$ for all $a\in \Gamma_0$.
    Then the following maps are upper semicontinuous:

    \begin{enumerate}
        \item Dimension of radical layering: $\raddim:\rep_d(\mathcal{A})\rightarrow \mathcal{P}$, that maps the representation $M$ to its sequence of dimensions of factors in radical series, $V\mapsto (\dim \rad^k V/\rad^{k+1} V)_k$.
        \item Dimension of socle layering: $\socdim:\rep_d(\mathcal{A})\rightarrow \mathcal{P}$, that maps the representation $V$ to its sequence of dimensions of factors in socle series, $V\mapsto (\dim \soc^{k+1} V/\soc^{k} V)_k$.
        \item The product of above maps: $\vartheta: \rep_d(\mathcal{A})\rightarrow \mathcal{P}\times \mathcal{P}$, given by
        $V\mapsto (\raddim(V),\socdim(V))$.
    \end{enumerate}
\end{defi}

These maps were previously used by~\cite{ClosuresVarietiesReps, IrrCompLocal}. For the proof of semicontinuity see~\cite[Observation 2.4]{IrrCompLocal}.

We say $\mathcal{P}$ is well partially ordered if it satisfies the descending chain condition and every subset $V\subset \mathcal{P}$ has at most finitely many minimal elements. By~\cite[Observation 2.2]{IrrCompLocal} if $\mathcal{P}$ is well partially ordered and $f:X\rightarrow \mathcal{P}$ is upper semicontinuous, then $f$ is generically constant on each irreducible subvariety $V\subseteq X$ and its generic value equals $f(U)=\min\{f(x)\; | \; x\in U\}$. This in particular applies to the above maps.

\section{Irreducibility}~\label{irr}

Fix a dimension vector $\underline{d}$. We wish to define a morphism of vector bundles whose kernel is $\rep(\underline{d}, \mathcal{A})$. To do this consider the following:

Let $\UT(k;\underline{d}\times \underline{d}')$ denote the set of block upper triangular matrices with blocks of dimensions $\underline{d}\times \underline{d'}$. That is matrices with the following block upper triangular form:

\begin{equation*}
    \begin{array}{c @{} cccc @{} l @{} l}
            &  d_{m-1}& d_{m-2}&\dots& d_{0} & & \\       
        \left( 
        \vphantom{\begin{array}{@{} c @{}}  0  \\ 0 \\ \vdots \\ 0 \end{array}} \right. &
        \begin{array}{@{} c @{}}  0  \\ 0 \\ \vdots \\ 0 \end{array} &
        \begin{array}{@{} c @{}} * \\ 0 \\ \vdots \\ 0 \end{array} & 
        \begin{array}{@{} c @{}} \dots \\ \ddots \\ \ddots \\ \dots \\ \end{array} &
        \begin{array}{@{} c @{}} * \\  * \\ \vdots \\ 0 \end{array} &
        \left. \vphantom{ 
        \begin{array}{@{} l @{}} * \\  * \\ \vdots \\ 0 \end{array}}
        \right) & 
        \begin{array}{@{} c @{}}
            d'_{m-1}  \\
            d'_{m-2} \\
            \vdots \\
            d'_{0}
         \end{array}
    \end{array}
\end{equation*}

\begin{defi}
    Let $\underline{d}=(d_i)_{i=0}^{m-1}$ be a dimension vector $d=\sum_i d_i$ and let $\underline{d}'=(d_i)_{i=1}^{m-1}$. Then define the trivial vector bundle $\pi: X\rightarrow Y$ with fibre $Z$, where

    \begin{align*}
        X&=\prod_{f\in \Gamma_1} \UT(k; \underline{d}(e(f))\times \underline{d}(s(f))),\\
        Y&=\prod_{f\in \Gamma_1} \UT(k; \underline{d}'(e(f))\times\underline{d}'(s(f))),\\
        Z&=\prod_{f\in \Gamma_1} \Mat(k; d(e(f))-d_0(e(f))\times d_0(s(f)))\\
    \end{align*}

    And $\pi$ is the natural projection $\pi:X\rightarrow Y$.
\end{defi}

There exists a natural inclusion $\rep(\underline{d}, \mathcal{A})\subseteq X$. Under this inclusion, $\rep(\underline{d},\mathcal{A})\subseteq \pi^{-1}(\rep(\underline{d}',\mathcal{A}))$. Indeed, for every $M\in\rep(\underline{d},\mathcal{A})$, $\pi(M)$ is clearly a subrepresentation of $\pi(M)$ and by lemma~\ref{lema_odp} its radical layering id $\underline{d}'$.

We wish to show that this is a kernel of some morphism of vector bundles. To define this morphism, let $(r_i)_{i\leq m}$ a set of homogeneous generators of $I$ over $k\Gamma^{+k}$, so $I=(r_1,\dots,r_s)+k\Gamma^{+k}$. Then we can write $r_i=\sum_j c_{ij}g_{ij}x_{j}$, for some elements $g_{ij}$ and edges $x_j$ with $s(x_{j})=s(r_i)$, $e(x_j)=s(g_{ij})$ and $e(g_{ij})=e(r_i)$. 

\begin{defi}
    Let $\mathcal{A}$ be an algebra as above and $\underline{d}$ a dimension vector. Then we define $W$ to be the vector space;

    \[W_{\underline{d}}=\prod_{i\leq s} \Mat(k; d(s(f_i))-d_0(s(f_i))\times d_0(e(f_i)))\]

    and $\Phi_{\underline{d}}$ to be the morphism of trivial vector bundles

\begin{align*}
    \Phi_{\underline{d}}:\rep(\underline{d}',\mathcal{A})\times Z&\rightarrow \rep(\underline{d}',\mathcal{A})\times W\\
    (M,(N_j)_{j\in \Gamma_1})&\mapsto \left(M, \left(\sum_{j} c_{ij}g_{ij}(M)\cdot N_j \right)_{i\leq m}\right)
\end{align*}

\end{defi}

Treating $\rep(\underline{d}',\mathcal{A})\times Z$ as a subset of $X$, we see that the kernel of this morphism is the set of representations of $k\Gamma$

\[P_j(f)=\begin{pmatrix}
    M_j(f) & N_j(f)\\
    0 & 0
\end{pmatrix}\]

with the $0$ block of dimension $d_0$, that satisfy the relations $f_i$. Indeed,

\[r_i((P_j)_j)=\begin{pmatrix}
    r_i((M_j)_j) & \sum_{j\leq n} c_{ij}g_{ij}(M(f))\cdot N_j(f) \\
    0 & 0
\end{pmatrix}\]

is $0$ iff $M\in \rep(\underline{d};,\mathcal{A})$ and $r_i(P_j)=0$

Hence $\rep(\underline{d}, \mathcal{A})$ is the set of representations in $\ker\Phi_{\underline{d}}$, such that the matrices in block $(0,1)$ do not have a common cokernel. As mentioned before this is an open condition, so $\rep(\underline{d}, \mathcal{A})$ is (possibly empty) open subset of $\ker\Phi_{\underline{d}}$. In particular, if $\rep(\underline{d}',\mathcal{A})$ is irreducible and $\ker\Phi_{\underline{d}}$ is a vector bundle, then $\rep(\underline{d},\mathcal{A})$ is irreducible as well. To determine if $\ker\Phi_{\underline{d}}$ is a vector bundle, it is enough to determine the dimensions of its fibres;

\begin{lemma}
    Let $\rho: E\rightarrow F$ be a morphism of vector bundles, such that $\dim\ker\rho_x$ has constant dimension. Then $\ker\rho$ is a vector bundle
\end{lemma}

\begin{proof}
    ~\cite[Proposition 1.7.2]{potier1997lectures}
\end{proof}

We can also determine the dimensions of fibres. To do this we fix the following notation: let $a\in \Gamma_0$. Then let $\Gamma_a$ denote the edges in $f\in \Gamma_1$ with $s(f)=a$ and $R_a\subseteq (r_1,\dots,r_s)$ denote the relations with start $s(r)=a$.  

\begin{lemma}~\label{dim_vlakna}
    With notation as above, let $M\in\rep(\underline{d}',\mathcal{A})$. Let $B(a)$ be the matrix

    \[B(a)=(c_{i,j}g_{ij})_{i\in R_a,\; j\in \Gamma_a}\]

    and let

    \[n(a)=\sum_{f\in\Gamma_a} d'(e(f))\]
    
    Then the dimension of $\ker\Phi_{\underline{d}}$ above the point $M$ equals

    \[\sum_{a\in\Gamma_0}\left(n(a)-\rank B(a)\right)\cdot d_0(a).
    \]
\end{lemma}

\begin{proof}
    By definition, $\ker(\Phi_{\underline{d}})_M$ is the vector space of all $|\Gamma_1|$-tuples of matrices $(N_{ij})_{j\in \Gamma_1}$, that satisfy the equations 

    \[\sum_{j} c_{ij}g_{ij}(M)\cdot N_j\]

    We can combine the relations with common start $a$ to get a matrix equation 

    \begin{equation*}    
    B(a)\begin{pmatrix}
        N_{j_1}\\
        \vdots\\
        N_{j_r}
    \end{pmatrix}
    =0.
    \end{equation*}

    But there exist matrices $P,Q\in \GL(k)$, such that

    \[B=P\begin{pmatrix}
        I_{r(a)} & 0\\
        0 & 0
    \end{pmatrix} Q\]

    is the decomposition of $B$ into a Smith normal form. Here $I_{r(a)}$ is the identity matrix of rank $r(a)$. Then a matrix $N$ solves the equation $BN=0$ iff $Q^{-1}N$ solves the equation 

    \[\begin{pmatrix}
        I_{r(a)} & 0\\
        0 & 0
    \end{pmatrix}X=0\]

    The space of solutions of this matrix equation has dimension $(n(a)-r(a))\cdot d_0(a)$.

    The result follows because the equations for $N_i$ with different starts are independent, so the solutions space is a direct sum of solutions spaces for each $a$.
\end{proof}

\begin{ex}
    Consider the algebra $\mathcal{A}=k\langle x_1,\dots,x_n\rangle /I$ where 

    \[I=(x_1,\dots, x_n)^m+\sum_{i=1}^n a_ix_i\]
    and $a_i=\sum c_{ij}x_j$ are linearly independent linear combinations of elements $x_1,\dots, x_n$. Let $M\in \rep(\underline{d},\mathcal{A})$ and let $M'=\rad M$. Then the $ij$-th block of the matrix $B$ in proof of proposition~\ref{dim_vlakna} equals $c_{ij}\varphi_{M'}(x_j)$.

    Further, since $a_i$ are linearly independent there exist some $d_{jk}\in k$, such that for all $j$,

    \[x_j=\sum_k d_{jk}a_k\]

    and hence

        \[\begin{pmatrix}
        \varphi_{M'}(x_1) & \dots & \varphi_{M'}(x_n)   \end{pmatrix}= B\cdot \left(\begin{pmatrix}
        d_{11}& \dots  & d_{1n}\\
        \vdots& \ddots & \vdots\\
        d_{n1}& \dots  & d_{nn}
    \end{pmatrix}\otimes I\right) \]

    Hence the image of matrix $B$ is exactly the sum of images of matrices $\varphi_{M'}(x_i)$, which is by definition equal to $\rad \mathcal{A}\cdot M'=\rad^2 M$. In particular, the dimension of fibres depends only on the radical layering of elements in $\rep(\layer, \mathcal{A})$, which is constant.
\end{ex}

The above example implies the following;

\begin{cor}\label{cover}
    Let $\mathcal{A}=k\langle x_1,\dots,x_n\rangle/I$, where 

    \[I=(x_1,\dots,x_n)^m+\left(\sum_i a_ix_i\right) \]

    and $(a_i)_i$ are linearly independent linear combinations of elements $x_1,\dots,x_n$. Then the varieties $\ker\Phi_{\underline{d}}$ are irreducible. In particular their open subvarieties $\rep(\underline{d},\mathcal{A})$ and images $\rep_{\underline{d}}(\mathcal{A})$ of those under the action of $\GL_d(k)$ are irreducible. Hence

    \[\rep_{d}(\mathcal{A})=\coprod_{\underline{d}}\rep_{\underline{d}}(\mathcal{A}) \]
    
    is a covering with irreducible quasi-affine varieties.
\end{cor}

It follows that all irreducible components of $\rep_{d}(\mathcal{A})$ are among the varieties $\overline{\rep_{\underline{d}}(\mathcal{A})}$. Dually one can prove the following

\begin{cor}\label{cover_soc}
    Let $\mathcal{A}=k\langle x_1,\dots,x_n\rangle/I$, where 

    \[I=(x_1,\dots,x_n)^m+\left(\sum_i a_ix_i\right) \]

    and $(a_i)_i$ are linearly independent linear combinations of elements $x_1,\dots,x_n$. Then the varieties $ \rep^*_{\underline{d}}(\mathcal{A})$ are irreducible.
    
    In particular

    \[\rep_{d}(\mathcal{A})=\coprod_{\underline{d}} \rep^*_{\underline{d}}(\mathcal{A}) \]
    
    is a covering with irreducible quasi-affine varieties.
\end{cor}

One could use a similar method to show that $\rep_{\underline{d}}(\mathcal{A})$ are irreducible for some other path algebras of quivers with relations. 

\begin{ex}
    If we consider the quiver

    
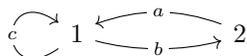
\[\begin{tikzcd}
	1 && 2
	\arrow["c"{description}, from=1-1, to=1-1, loop, in=145, out=215, distance=10mm]
	\arrow["b"{description}, curve={height=6pt}, from=1-1, to=1-3]
	\arrow["a"{description}, shift right, curve={height=6pt}, from=1-3, to=1-1]
\end{tikzcd}\]

    with relations $I=(a,b,c)^n+I'$ with generated by $I'$ any of the following relations

    \[\{ab+c^2,bc, bab+bc^2, ba\} \]

    Then the varieties $\rep(\underline{d},k\Gamma/I)$ are vector bundles and yield irreducible cover. However, for $I'=(bab+bc^2)$ the varieties $\rep(\underline{d},k\Gamma/I)$ are not always vector bundles and hence are not always irreducible.
\end{ex}

But it does not work for all path algebras with relations.

\begin{ex}
    For algebra $k[x,y]/(x^3,y^2)$ the $\ker\Phi_M$ does not have constant dimensions, so $\rep(\underline{d},k\Gamma/I)$ are not vector bundles and hence not irreducible.
\end{ex}

In such cases, other filtrations might work better. See, for example,~\cite{BOBINSKI2019624}.

\section{Possible radical layerings}\label{possible_layerings}

From now on we consider the algebra $\mathcal{A}=k\langle x_1,\dots,x_n\rangle/I$, where 

    \[I=(x_1,\dots,x_n)^3+\left(S\right) \]

and $S=\sum_i a_{-j}x_ix_j$, with $\det(a_{ij})\neq 0$.
By corollaries~\ref{cover} and~\ref{cover_soc} the irreducible components are at the same time closures of varieties $\rep_{\underline{d}}(\mathcal{A})$ and $\rep^*_{\underline{d}}(\mathcal{A})$.

When studying possible radical layerings and their generic socle layering, we may without loss of generality assume $S=\sum_i x_i^2$. Indeed each representation in $\rep(\underline{d},\mathcal{A})$ has a block upper triangular form

\[\begin{pmatrix}
    0 & A_i & B_i\\
    0 &  0 & C_i\\
    0 & 0 & 0
\end{pmatrix}_i\]

such that $\sum_{ij} c_{ij} A_jC_i=0$. But because $a_i$ are linearly independent the mapping 

\[(A_i)_i\mapsto (\sum_i c_{ij}A_i)\]

induces an isomorphism $\rep(\underline{d},\mathcal{A})\cong \rep(\underline{d},\mathcal{A}')$ where 

\[A'=k\langle x_1,\dots,x_n\rangle/((x_1,\dots,x_n)^3+\left(S'\right))\]

Where $S'=\sum_i x_i^2$
This isomorphism also clearly preserves radical layerings of representations.

\begin{thm}\label{rep_layer_exists_iff}
    Let $\mathcal{A}$ be as above and let $\underline{d}=(d_0,d_1,d_2)$ be a dimension vector. Then $\rep(\underline{d},\mathcal{A})$i s nonempty iff

    \begin{align*}
        d_1\leq n\cdot d_0,\\
        n\cdot d_2\leq (n^2-1)\cdot d_1
    \end{align*}
\end{thm}

\begin{rem}\label{soc_layer_exists_iff}
    Dually one can show that with same notation $\rep^*(\underline{d},A)$ is nonempty iff

    \begin{align*}
        d_1\leq n\cdot d_0,\\
        n\cdot d_2\leq (n^2-1)\cdot d_1
    \end{align*}
\end{rem}

The first inequality is trivial because the dimension of the sum of images of $n$ matrices from $\Mat(k;{d_1\times d_0})$ cannot exceed $n\cdot d_0$.

Further, it is enough to prove the second inequality in the case $A=(A_1,\dots, A_n)$ is indecomposable. Indeed if there exists a basis such that all $A_i$ have a block diagonal form

\[A_i=\begin{pmatrix}
    A_i^{(1)} & \dots & 0\\
    \vdots & \ddots & \vdots \\
    0 & \dots & A_i^{(m)}\\
\end{pmatrix}\]

with common dimensions of blocks, then we can write

\[C_i=\begin{pmatrix}
    C_i^{(1)}\\
    \vdots\\
    C_i^{(m)}\\
\end{pmatrix}\]

and the relation $S$ becomes

\[S=\begin{pmatrix}
    \sum_i A_i^{(1)}C_i^{(1)}\\
    \vdots\\
    \sum_i A_i^{(m)}C_i^{(m)}\\
\end{pmatrix}\]

Now each block has to satisfy the inequality and hence the whole matrix has to as well.

\begin{lemma}\label{izjemna upodobitev}
    Let $A=(A_1,\dots,A_n)\in\Mat(k;{d_2\times d_1})$ be indecomposable. Suppose further that $\dim \sum \im A_i=d_2$ and $A$ are not simultaneously equivalent to 

    \[B=\begin{pmatrix}
        \begin{pmatrix}
            1 \\
            0 \\
            \vdots\\
            0
        \end{pmatrix} & \begin{pmatrix}
            0 \\
            1 \\
            \vdots\\
            0
        \end{pmatrix}& \dots & \begin{pmatrix}
            0 \\
            0 \\
            \vdots\\
            1
        \end{pmatrix}
    \end{pmatrix}\]

    that is there exist no invertible matrices $P,Q$ such that $PA_iQ^{-1}=B_i$ for all $i$.

    Then the following inequality holds

    \begin{equation*}
    n\cdot d_2\leq (n^2-1)d_1       
    \end{equation*}

\end{lemma}

\begin{proof}
    $A$ as above forms an indecomposable representation of $n$-Kronecker quiver. By~\cite[Theorem 1]{Kac1980} its dimensions $(d_1,d_2)$ have to satisfy the inequality 
    
    \[q(d_1,d_2)=d_1^2+d_2^2-nd_1d_2\leq 1\]

    Further, if the equality condition is satisfied there exists exactly one class of such $n$-tuples with dimensions $(d_1,d_2)$, up to simultaneous equivalence.

    To check that inequality is satisfied we consider two cases. If $n\geq d_1$, then the line $d_2=\frac{n^2-1}{n}d_1$ lies above the hyperbola $q(d_1,d_2)=1$, so the inequality is clear.

    If $n< d_1$ on the other hand, we can estimate $nd_1-1< \frac{n^2-1}{n}d_1$. Now, since $nd_1$ is an integer, any integral point above the line $d_2=nd_1-1$ must also lie above $d_2=nd_1$. But this again lies above the hyperbola $q(d_1,d_2)=1$ unless $(d_1,d_2)=(1,n)$. 

    Hence the inequality is satisfied in all cases except for $(d_1,d_2)=(1,n)$. In this case, however, $q(d_1,d_2)=1$, so, up to simultaneous similarity, there exists only one such tuple of matrices, which is exactly the exceptional tuple in the lemma.
\end{proof}

\begin{figure}[htb]
    \begin{subfigure}[b]{0.49\textwidth}
    \resizebox{\textwidth}{!}{
    \definecolor{modra}{RGB}{71,71,255}
\definecolor{rdeca}{RGB}{255,71,71}

\begin{tikzpicture}[line cap=round,line join=round,>=triangle 45,x=1.0cm,y=1.0cm]
\begin{axis}[
x=1.0cm,y=1.0cm,
axis lines=middle,
ymajorgrids=true,
xmajorgrids=true,
xmin=0.0,
xmax=10.0,
ymin=0.0,
ymax=10.0,
xtick={0.0,1.0,...,10.0},
ytick={0.0,1.0,...,10.0},]
\addplot [modra,name path=A, domain=0:10] {(3*x+sqrt( 9*x^2-4*(x^2-1)))/2};

\addplot [modra,name path=B, domain=0:10] {(3*x-sqrt( 9*x^2-4*(x^2-1)))/2};

\addplot [modra!60!white, opacity=0.3] fill between [of=A and B];

\node [circle, fill=modra] at (1,0) {};
\node [circle, fill=modra] at (1,1) {};
\node [circle, fill=modra] at (1,2) {};
\node [circle, fill=modra] at (2,5) {};
\node [circle, fill=modra] at (3,8) {};

\node [circle, fill=rdeca] at (1,3) {};
\end{axis}
\end{tikzpicture}}
    \end{subfigure}
    \begin{subfigure}[b]{0.49\textwidth}
    \resizebox{0.968\textwidth}{!}{
    \definecolor{modra}{RGB}{71,71,255}
\definecolor{rdeca}{RGB}{255,71,71}

\begin{tikzpicture}[line cap=round,line join=round,>=triangle 45,x=1.0cm,y=1.0cm]
\begin{axis}[
x=1.0cm,y=1.0cm,
axis lines=middle,
ymajorgrids=true,
xmajorgrids=true,
xmin=0.0,
xmax=16.0,
ymin=0.0,
ymax=16.0,
xtick={0.0,1.0,...,16.0},
ytick={0.0,1.0,...,16.0},]
\addplot [modra,name path=A, domain=0:16] {(4*x+sqrt( 16*x^2-4*(x^2-1)))/2};

\addplot [modra,name path=B, domain=0:16] {(4*x-sqrt( 16*x^2-4*(x^2-1)))/2};

\addplot [modra!60!white, opacity=0.3] fill between [of=A and B];

\node [circle, fill=modra] at (1,0) {};
\node [circle, fill=modra] at (1,1) {};
\node [circle, fill=modra] at (1,2) {};
\node [circle, fill=modra] at (1,3) {};
\node [circle, fill=modra] at (2,7) {};
\node [circle, fill=modra] at (3,11) {};
\node [circle, fill=modra] at (4,15) {};

\node [circle, fill=rdeca] at (1,4) {};
\end{axis}
\end{tikzpicture}}
    \end{subfigure}
    \caption{Set of roots of $q$ and the generators of its $\mbb{N}$ linear closure. for $3$-tuples and $4$-tuples of matrices. The dimension vector that does not appear by lemma~\ref{izjemna upodobitev} is marked in red.}
    \label{fig:kronecker2}        
\end{figure}
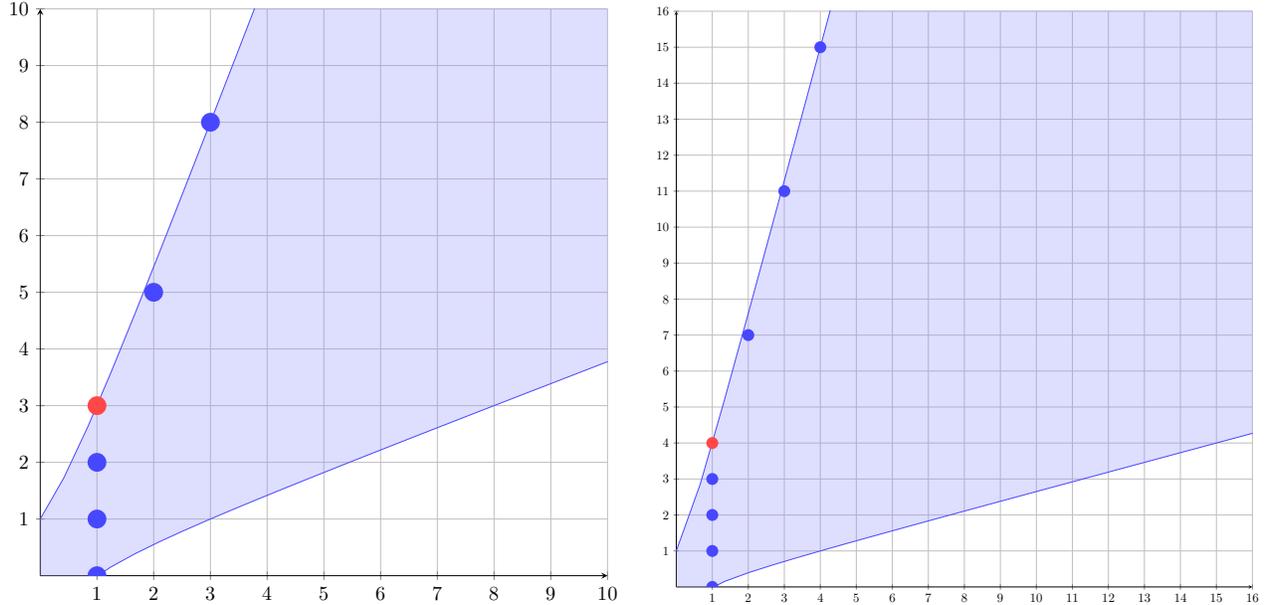

For the exceptional tuple in the above lemma, there exists no nonzero tuple $C=(C_1,\dots, C_n)$ of matrices, such that $A$ and $C$ satisfy the relation $S$. Indeed in this case each $C_i$ is a $1\times n$ matrix and 

\[\sum A_iC_i=\begin{pmatrix}
    C_1\\
    \vdots\\
    C_n
\end{pmatrix}\]

With this, we have shown that the condition in theorem~\ref{rep_layer_exists_iff} is necessary. We are left to show that it is sufficient. We provide the construction in the cases $(d_1,d_2)$ is equal to one of

\begin{align*}
    (1,0),\dots,(1,n-1) & \text{ in lemma~\ref{lema_dim_1},}\\
(2,2n-1),(3n,3n-1)\dots, (n,n^2-1)   & \text{ in lemma~\ref{lema_dim_gt1}.}
\end{align*}

This is sufficient because we can then take the direct sum of such construction to get a construction for other $(d_1,d_2)$. All $(d_1,d_2)$ that satisfy the inequality $nd_2\leq (n^2-1)d_1$ can be obtained this way by the following lemma~\ref{baza korenov}. The set of dimensions of indecomposable modules along with the above generators and the exceptional dimension vector are shown in figure~\ref{fig:kronecker2}.

\begin{lemma}\label{baza korenov}
    Every pair $(d_1,d_2)$, that satisfies the inequality $nd_2\leq (n^2-1)d_1$ can be written as a positive linear combination of pairs

    \[(1,0),\dots,(1,n-1),(2,2n-1),(3,3n-1),\dots,(n,n^2-1)\]
\end{lemma}

\begin{proof}
    For the sake of contradiction suppose $(d_1,d_2)$ is the minimal counterexample (with respect to $d_1+d_2$). Then $d_2\leq n^2-1$ as otherwise $d_1> n$ and $(d_1-n,d_2-n^2+1)$ is a smaller counterexample. Similarly $d_1\leq n$, as otherwise $(d_1-1,d_2)$ would be a smaller counterexample.

    In the proof of lemma~\ref{izjemna upodobitev}, we have shown that for such pairs $d_2\leq nd_1-1$. 
    
    If the equality is satisfied or $d_1=1$, the pair is in our set. Otherwise, $d_2\geq n$ and the inequality is also satisfied for $(d_1-1,d_2-n+1)$ or $d_2\leq n-1$ and the inequality is satisfied for $(1,d_2)$. Therefore there always exists a smaller counterexample. So there is no smallest counterexample, and the lemma is true for all pairs $(d_1,d_2)$.
\end{proof}

In the following constructions, we denote the tuple $A=(A_1,\dots, A_n)$ of matrices by a formal sum $\sum_i\lambda_iA_i$.

\begin{lemma}\label{lema_dim_1}
    Let $m\leq n-1$. Then for all $d_0\geq 1$, there exists a representation of $\mathcal{A}$ with radical layering $(d_0,1,m)$.
\end{lemma}

\begin{proof}
    The pair

    \[A=\begin{pmatrix}
        \lambda_1\\
        \lambda_2\\
        \vdots\\
        \lambda_m
    \end{pmatrix}, \qquad C=\begin{pmatrix}
        
\lambda_{m+1} & 0 & \dots & 0    \end{pmatrix}\]

    gives such a representation. In case $m=0$ this gives us a matrix

    \[\begin{pmatrix}
        0 & \lambda_1 & \dots & 0\\
        0 & 0 & \dots & 0\\
        \vdots & \vdots & \ddots & \vdots\\
        0 & 0 & \dots & 0
    \end{pmatrix}\]
\end{proof}

\begin{lemma}\label{lema_dim_gt1}
    Let $2\leq m\leq n$. Then for all $d_0\geq 1$ there exists a representation of $\mathcal{A}$ with radical layering $(d_0,m,m\cdot n-1)$.
\end{lemma}

\begin{proof}
    Consider the following representation. Let

    \[A=\begin{pmatrix}
        I_k \cdot \lambda_1\\
        \vdots\\
        I_k \cdot \lambda_{n-1}\\
        \end{pmatrix}+A_n\cdot \lambda_n\]

    where $A_n$ is the matrix

    \[A_n=\begin{pmatrix}
        0_{(m\cdot (n-1)-1)\times m }\\
        I_k
    \end{pmatrix}+\sum_{i=0}^{k-3} E_{(i+2+m\cdot i,1)}\]

    where $I_m$ is the $m\times m$-identity matrix and $E_{i,j}$ is the matrix with $1$ in component $(i,j)$ and $0$ otherwise. Here in the case $m=2$, the last sum is empty.

    Further, let

    \[C=\begin{pmatrix}
        -\lambda_n & 0 & \dots & 0 \\
        \lambda_1 & \\
        \vdots &  \vdots & & \vdots \\
        \lambda_{m-2}& \\
        \lambda_{n-1} & 0 & \dots & 0 
    \end{pmatrix}\]

Here note that the last coefficient is $\lambda_{n-1}$ and not $\lambda_{m-1}$.

This gives us a representation of algebra $\mathcal{A}$ as needed and it is not difficult to check that it has the desired radical layering.
\end{proof}

\begin{rem}
    The definition of matrix $A$ is slightly technical, here are examples for the cases $m=2,3,4$ and $n=4$

    \small{\[\begin{pmatrix}
        \lambda_1 & \cdot \\
        \cdot  & \lambda_1\\
        \lambda_2 & \cdot \\
        \cdot  & \lambda_2\\
        \lambda_3 & \cdot \\
        \lambda_4 & \lambda_3\\
        \cdot & \lambda_4
    \end{pmatrix},\qquad \qquad \begin{pmatrix}
        \lambda_1 & \cdot & \cdot\\
        \lambda_4 & \lambda_1 & \cdot\\
        \cdot & \cdot & \lambda_1\\
        \lambda_2 & \cdot & \cdot\\
        \cdot & \lambda_2 & \cdot\\
        \cdot & \cdot & \lambda_2\\
        \lambda_3 & \cdot & \cdot\\
        \cdot & \lambda_3 & \cdot\\
        \lambda_4 & \cdot & \lambda_3\\
        \cdot & \lambda_4 & \cdot\\
        \cdot & \cdot & \lambda_4
    \end{pmatrix}, \qquad \qquad \begin{pmatrix}
        \lambda_1 & \cdot & \cdot & \cdot\\
        \lambda_4 & \lambda_1 & \cdot & \cdot \\
        \cdot & \cdot & \lambda_1 & \cdot\\
        \cdot & \cdot & \cdot & \lambda_1\\
        \lambda_2 & \cdot & \cdot & \cdot\\
        \cdot & \lambda_2 & \cdot & \cdot \\
        \lambda_4 & \cdot & \lambda_2 & \cdot\\
        \cdot & \cdot & \cdot & \lambda_2\\
        \lambda_3 & \cdot & \cdot & \cdot\\
        \cdot & \lambda_3 & \cdot & \cdot \\
        \cdot & \cdot & \lambda_3 & \cdot\\
        \lambda_4 & \cdot & \cdot & \lambda_3\\
        \cdot & \lambda_4 & \cdot & \cdot \\
        \cdot & \cdot & \lambda_4 & \cdot\\
        \cdot & \cdot & \cdot & \lambda_4\\
    \end{pmatrix}\]}
\end{rem}

\section{Containments}\label{incl}

We are left with the task to determine which of the varieties $\rep_{\underline{d}}(\mathcal{A})$ are maximal for inclusion. This can be split into two parts. We need to demonstrate the containment for non-maximal varieties and separate the maximal ones.
To separate the maximal varieties, we compute their generic socle and radical layering. For ease of notation, from now on let
\begin{align*}
    \socdim(\underline{d})&=\socdim(\rep_{\underline{d}}(\mathcal{A})),\\
    \raddim(\underline{d})&=\socdim(\rep^*_{\underline{d}}(\mathcal{A}))\\    
\end{align*}

Before we continue, recall the following lemma~\cite[Lemma 2.4]{IrrCompLocal} that shows how $\raddim$ and $\socdim$ complement each other.

\begin{lemma}\label{socdim_raddim_complement}
    Let $A=k\langle x_1,\dots, x_n\rangle/I$ be an algebra with $J^n\subseteq I\subseteq J^2$. Let $M$ be a $d$-dimensional $A$-module. Then
    \begin{itemize}
        \item If $\raddim M=(d_i)_{i=0}^{n-1}$, then

        \[(d_{n-i})_{i=0}^{n-1}\preceq \socdim M\]

        and if the equality is satisfied, then $(\raddim M,\socdim M)$ is minimal in $\vartheta_{\rep_d}(A)$
            \item If $\socdim M=(d_i)_{i=1}^{n}$, then

        \[(d_{n-i})_{i=1}^n\preceq \raddim M\]

        and if the equality is satisfied, then $(\raddim M,\socdim M)$ is minimal in $\vartheta_{\rep_d}(A)$
    \end{itemize}
\end{lemma}

We also have the following necessary conditions for irreducible components:

\begin{lemma}\label{lemma_min_soc_rad}
    Let $\underline{d}$ be a dimension vector, such that $\overline{\rep_{\underline{d}}(\mathcal{A})}$ is an irreducible component of $\rep_d(\mathcal{A})$ and let

    \begin{align*}
        \underline{d}'&=\socdim(\underline{d}),\\
        \underline{d}''&=\raddim(\underline{d}')\\
    \end{align*}

    Then $\underline{d}''=\underline{d}'$.
\end{lemma}

\begin{proof}
    Suppose this is not true. Because $\socdim$ is upper semicontinuous $\rep^*_{\underline{d}'}(\mathcal{A})\cap \overline{\rep_{\underline{d}}(\mathcal{A})}$ is open in $\overline{\rep_{\underline{d}}(\mathcal{A})}$. But then, since $\overline{\rep_{\underline{d}}(\mathcal{A})}$ is irreducible,

    \[\overline{\rep^*_{\underline{d}'}(\mathcal{A})}\supseteq \overline{\rep_{\underline{d}'}(\mathcal{A})}\]

    similarly

    \[\overline{\rep_{\underline{d}''}(\mathcal{A})}\supseteq \overline{\rep_{\underline{d}'}^*(\mathcal{A})}\]

    and hence     
    
    \[\overline{\rep_{\underline{d}''}(\mathcal{A})}\supseteq \overline{\rep_{\underline{d}}(\mathcal{A})}\]

    From maximality of $\overline{\rep^*_{\underline{d}'}(\mathcal{A})}$ and the fact that $\raddim$ is generically constant on irreducible varieties it follows, that $\underline{d}''=\underline{d}'$.
\end{proof}

As before every representation $M=(M_1,\dots,M_n)$  in $\rep(\underline{d},\mathcal{A})$ has a block upper triangular form induced by $\underline{d}$;

\begin{equation}\label{form}
    M_i=\begin{pmatrix}
    0 & A_i & B_i\\
    0 & 0 & C_i \\
    0 & 0 & 0
\end{pmatrix}  
\end{equation}

\begin{defi}
    With notation as above consider the maps $k_0,k_1:\rep(\underline{d},\mathcal{A})\rightarrow \mbb{N}$.

    \begin{align*}
        h_0:M&\mapsto \dim \bigcap_i \ker C_i,\\
        h_1:M&\mapsto \dim \bigcap_i \ker A_i
    \end{align*}
    
\end{defi}

These maps are upper semicontinuous because they are determined by ranks of matrices 

\[\begin{pmatrix}
    C_1\\
    \vdots\\
    C_n
\end{pmatrix} \qquad \text{and} \qquad \begin{pmatrix}
    A_1\\
    \vdots\\
    A_n
\end{pmatrix}\]

\begin{defi}
    Dually we also consider the maps $h'_0,h'_1:\rep^*(\underline{d},\mathcal{A})\rightarrow \mbb{N}$.

    \begin{align*}
        h'_0:M&\mapsto d_1-\dim \sum_i \im A_i,\\
        h'_1:M&\mapsto d_2-\dim \sum_i \im C_i
    \end{align*}
    
\end{defi}

These are again upper semicontinuous because they are given by ranks of matrices

\[\begin{pmatrix}
    A_1 & \dots & A_n
\end{pmatrix} \qquad \text{and} \qquad \begin{pmatrix}
    C_1 & \dots & C_n
\end{pmatrix}\]

If there is some representation $M\in\rep(\underline{d},\mathcal{A})$, such that $h_0(M)=h_1(M)=0$, then the radical filtration for $M$ is also its socle filtration\footnote{Note that this means that their dimension vectors are flipped}, by the dual version of lemma~\ref{lema_odp}. Hence by~\ref{socdim_raddim_complement}, $\vartheta(\rep(\underline{d},\mathcal{A}))$ is minimal and since no two potential components have the same generic values of $\vartheta$, $\rep_{\underline{d}}(\mathcal{A})$ must be a component.

Hence it is enough to consider what happens when one of $h_i$ is not equal to $0$. By semicontinuity this means that $h_i(M)\neq 0$,  for all $M\in\rep(\underline{d},\mathcal{A})$.

First consider the case $h_0(\rep(\underline{d},\mathcal{A}))\neq 0$. Then every representation $M\in\rep(\underline{d},\mathcal{A})$ is isomorphic to some matrix with block upper triangular form

\[M_i= \bordermatrix{
        &d_2&d_1&1&d_0-1\cr
    d_2 & 0 & A_i & * & B_i' \cr
    d_1 & 0 & 0 & 0 & C_i' \cr
    d_0 & 0 & 0 & 0 & 0 \cr
    }\]

Let $\underline{d}'=(d_0-1,d_1+1,d_2)$. Then $M$ lies in $\ker\Phi_{\underline{d'}}$. In particular, if $\rep_{\underline{d}'}(\mathcal{A})$ is nonempty, $M$ lies in its closure.

We have shown

\begin{lemma}
    Let $\underline{d}=(d_0,d_1,d_2)$ be a dimension vector of a radical layering, such that $h_0(M)\neq 0$ for all $M\in\rep(\underline{d},\mathcal{A})$. If variety $\rep(\underline{d}',\mathcal{A})$, where $\underline{d}'=(d_0+1,d_1-1,d_2)$ is nonempty, then  $\rep_{\underline{d}}(\mathcal{A})$ is not an irreducible component.
\end{lemma}

To use this we need to compute $h_0(\rep(\underline{d},\mathcal{A}))$:

\begin{prop}\label{value_k0}
    Let $\underline{d}=(d_0,d_1,d_2)$ be a dimension vector of radical filtration, such that $\rep(\underline{d},\mathcal{A})$ is nonempty. Then

    \[h_0(\rep(\underline{d},\mathcal{A}))=\max(d_0-(n\cdot d_1-d_2),0)\]
\end{prop}

To prove this we need the following lemma;

\begin{lemma}\label{relation}
    Let $A=(A_i)_{i\leq n}\in \Mat(k;{a\times b})^n$ be a $n$-tuple of matrices, such that $\dim \sum_i \im A_i=a$ and let $V$ be the space of all $n$-tuples of vectors $v=(v_i)_n$ in $k^b$ that satisfy the equation

    \[\sum A_iv_i=0\]

    Then 

    \[\dim V=\max (n\cdot b-a,0)\]
\end{lemma}

\begin{proof}
    The above condition gives us a linearly independent system of $a$ equations for the vector $\Tilde{v}=(v_{i,j})_{i,j}$.
\end{proof}

\begin{proof}[Proof of proposition~\ref{value_k0}]
    Consider tuples of upper triangular matrices as in~\ref{form}. If the tuple $A$ has a trivial common cokernel, the relation $S$ induces a matrix equation for the columns of $C$. Hence, by lemma~\ref{relation}, the space of possible columns has dimension $n\cdot d_1-d_2$. ($nd_1\geq d_2$ by theorem~\ref{rep_layer_exists_iff}). Hence there can be at most $n\cdot d_1-d_2$ linearly independent columns.

    Since $\rep(\underline{d},\mathcal{A})$ is non-empty, there exists some set of linearly independent columns such that the matrices $C$ have trivial common cokernel. We can extend this to a maximal set of linearly independent columns. If $d_0\geq n\cdot d_1-d_2$ we further extend by $0$ columns to get a representation $M\in \rep(\underline{d},\mathcal{A})$ with $h_0(M)=d_0-(n\cdot d_1-d_2)$. Otherwise, take a subset of these with $d_0$ elements to get a representation with $h_0(M)=0$.
\end{proof}

Suppose now that $\underline{d}=(d_0,d_1,d_2)$ is a dimension vector, such that $\rep(\underline{d},\mathcal{A})$ is an irreducible component with $h_0(\rep(\underline{d},\mathcal{A}))\neq 0$. This means that for $\underline{d'}=(d_0+1,d_1-1,d_2)$, the variety $\rep(\underline{d'},\mathcal{A})$ is empty. By theorem~\ref{rep_layer_exists_iff} for $\underline{d}$ and using the fact that one column of $C$ is linearly dependent on the others, we get.

\[d_1\leq n\cdot (d_0-1) \qquad \text{and} \qquad n\cdot d_2\leq (n^2-1)\cdot d_1\]
Similarly applying the theorem~\ref{rep_layer_exists_iff} at least one of the following inequalities must be satisfied

\[d_1+1> n\cdot (d_0-1), \qquad n\cdot d_2 > (n^2-1)\cdot (d_1+1) \]

The second inequality leads to a contradiction

\[n\cdot d_2> (n^2-1)\cdot (d_1+1)>(n^2-1)\cdot d_1\geq n\cdot d_2\]

Hence the first must be satisfied. But then

\[d_1\leq n\cdot (d_0-1)< d_1+1\]

Hence $d_1=n\cdot (d_0-1)$. Further since $h_0(\rep(\underline{d},\mathcal{A})\neq 0$ lemma~\ref{value_k0} implies

\[1\leq d_0-n\cdot d_1+d_2\]

or equivalently

\[(n-1/n)d_1\leq d_2\]

Since the opposite inequality also holds, there exists some $a$, such that $\underline{d}=(a,n(a-1),(n^2-1)(a-1))$.

To prove that this dimension vector indeed induces irreducible components, check that the matrices with the above radical layering can be written in the form

\[\begin{pmatrix}
    0 & \begin{pmatrix}
        P & & 0  \\
         & \ddots & \\
         0 & & P
    \end{pmatrix} & B\\
    0 & 0 & \begin{pmatrix}
        0 & Q & & 0\\
        \vdots & & \ddots & \\
        0 & 0 & & Q
    \end{pmatrix}\\
    0 & 0 & 0
\end{pmatrix}\]

where $P, Q$ are indecomposable $n$-tuples of matrices with dimensions $(n^2-1)\times n$ and $n\times 1$ respectively. Now since the ratio of dimensions of the upper middle block $\bigoplus P$ is $\frac{n^2-1}{n}<n$, we can choose the first column of $B$ to be linearly independent of the columns of $\bigoplus P$. Then one can check that the socle layering for this representation is $((n^2-1)(a-1),n(a-1)+1,a-1)$. Hence it is enough to check that the pair

\[(\underline{d},\underline{d}')=\Big((a,n(a-1),(n^2-1)(a-1)),((n^2-1)(a-1),n(a-1)+1,a-1)\Big)\]

is minimal in $\vartheta(\rep_d(\mathcal{A}))$. By lemma~\ref{socdim_raddim_complement} for the given radical layering, the socle layering is greater than $((n^2-1)(a-1),n(a-1), a)$.
But, by assumption $k_0(\underline{d})=1$, so it must be at least $(n^2-1)(a-1)),((n^2-1)(a-1),n(a-1)+1,a-1)$. Similarly by lemma~\ref{socdim_raddim_complement} the radical layering must be at least $(a-1,n(a-1),(n^2-1)(a-1))$. But this does not satisfy the condition from theorem~\ref{rep_layer_exists_iff}. Hence the smallest possible radical layering is $(a,n(a-1),(n^2-1)(a-1))$. Hence the above pair is minimal, and $\overline{\rep(\underline{d}, \mathcal{A})}=\overline{\rep^*(\underline{d}',\mathcal{A})}$ is an irreducible component.

We are left with the case $h_1\neq 0$. As before we compute the generic value of $h_1$

\begin{lemma}
    Let $\underline{d}=(d_0,d_1,d-2)$ be a dimension vector of radical layering, such that $\rep(\underline{d},\mathcal{A})$ is nonempty. Then

    \[h_1(\rep(\underline{d},\mathcal{A}))=\max(d_1-n\cdot d_2,0)\]
\end{lemma}

\begin{proof}
    Clearly 

    \[h_1(M)=\dim \ker \begin{pmatrix}
        A_1\\
        \vdots\\
        A_n
    \end{pmatrix}=d-\rank \begin{pmatrix}
        A_1\\
        \vdots\\
        A_n
    \end{pmatrix}\geq \max (d_1-nd_2,0)\]

To show the other inequality, it is enough to construct a representation with given $h_1$. In the case $d_1-n\cdot d_2\geq 0$ consider the representation

    \[M=\begin{pmatrix}
        0 & A & 0\\
        0 & 0 & C\\
        0 & 0 & 0
    \end{pmatrix}\]

    with

   \[A=\begin{pmatrix}
        0_n& \dots & 0_n & 0_d & I & & 0 \\
        \vdots& \ddots & \vdots & \vdots &   & \ddots &  \\
        0_n& \dots & 0_n & 0_d & 0 & & I \\
    \end{pmatrix} \qquad \text{and} \qquad C=\begin{pmatrix}
        P & & & & &  & & 0\\
          & \ddots & & & & & & \vdots \\
          & & P & & & & &  \vdots\\
          & & & Q & & & & \vdots\\
          & & & & P & & & \vdots\\
          & & &  & & \ddots & & \vdots\\
          & & & &  & & P & 0 \\
    \end{pmatrix}\]

    where $I=(\lambda_1,\dots,\lambda_n)$ is an indecomposable $n$-tuple with dimension $1\times n$ and $0_n$ are zero matrice of dimension $1\times i$, where $d\equiv d_1 \pmod{n}$ and

    \[P=\begin{pmatrix}
        \lambda_1\\
        \vdots\\
        -(n-1)\lambda_n
    \end{pmatrix} \qquad \text{or} \qquad P=\begin{pmatrix}
        -\lambda_1\\
        \vdots\\
        -(n-3)\lambda_n
    \end{pmatrix}\]

    and $Q$ is an indecomposable $n$-tuple of matrices with dimension vector $(1,d)$. Here we choose $P$ such that the last row is nonzero.
    One can check that $M$ is the desired matrix. Similarly in case $d_2\leq d_1\leq n\cdot d_2$ we can use similar representation $M$ with

        \[A=\begin{pmatrix}
        I \\
        & \ddots \\
        & & I \\
        & & & I' \\
        & & & & \ddots\\
        & & & & & I'
    \end{pmatrix} \qquad \text{and} \qquad  C=\begin{pmatrix}
        P \\
        & \ddots \\
        & & P \\
        & & & P' \\
        & & & & \ddots\\
        & & & & & P'
    \end{pmatrix}\]

    where $I,I'$ are indecomposable tuples with dimensions $1\times l$ and $1\times (l+1)$, where $m$ and the number of tuples $I,I'$ is chosen such that dimension of the tuple $A$ is $d_1\times d_2$
    
    \[I=\begin{pmatrix}
        \lambda_1 & \dots & \lambda_l
    \end{pmatrix} \qquad \text{and} \qquad I'=\begin{pmatrix}
        \lambda_1 & \dots & \lambda_{l-1}
    \end{pmatrix}.\]

    and $P,P'$ are the matrices
    
       \[P=\begin{pmatrix}
        \lambda_1\\
        \vdots\\
        -(l-1)\lambda_l
    \end{pmatrix} \qquad \text{and} \qquad P'=\begin{pmatrix}
        \lambda_1\\
        \vdots\\
        -(l)\lambda_{l+1}
    \end{pmatrix}.\]    

    If $l-1=0$ or $l=0$ we use the same trick as before to avoid a column being $0$.
    Finally if $d_1\leq d_2$ we can let

     \[A=\begin{pmatrix}
        I \\
        & \ddots \\
        & & I \\
        & & & I' \\
        & & & & \ddots\\
        & & & & & I'
    \end{pmatrix} \qquad \text{and} \qquad  C=\begin{pmatrix}
        P \\
        & \ddots \\
        & & P \\
        & & & P' \\
        & & & & \ddots\\
        & & & & & P'
    \end{pmatrix}\]

    where $I,I'$ and $P,P'$ are the indecomposable $n$-tuples with dimension vector $(l^2-1)\times l$, $((l+1)^2-1)\times (l+1)$ and $l\times 1$, $(l+1)\times 1$ respectively as in lemma~\ref{lema_dim_gt1} for appropriate $l$.
\end{proof}

With this, we can compute the generic socle layering for any radical layering. Let $\underline{d}=(d_0,d_1,d_2)$ and let $\underline{d}'=\socdim(\underline{d})=(d_0',d_1',d_2')$.
To compute generic socle layering, it is enough to compute $d_0'$ and $d_1'$.

We may also assume $h_1(\rep(\underline{d}, \mathcal{A}))\neq 0$ since we have determined all irreducible components with $h_1=0$ (namely all the dimension vectors with $h_0=0$ and the exceptional dimension vector with $h_0=1$).

\begin{lemma}
    With notation as above

    \[d_0'=d_2+(d_1-n\cdot d_2)\]
\end{lemma}

\begin{proof}
    We must determine the generic dimension common kernel of matrices in $M$. Since each representation $M$ is isomorphic to one with first $h_1(M)$ columns in $A$ equal to $0$, it must have the dimension of the common kernel at least $d_2+h_1(M)$.

    Conversely by proof of proposition~\ref{value_k0}, there exists some representation 

    \[M=\begin{pmatrix}
        0 & A & B \\
        0 & 0 & C\\
        0 & 0 & 0
    \end{pmatrix}\]

    in $\rep(\underline{d},\mathcal{A})$ such that $C$ has trivial common kernel. Hence if

    \[\begin{pmatrix}
        x\\
        y\\
        z\\
    \end{pmatrix}\in \ker M\]

    then $z\in \ker C=\{0\}$. So the common kernel of $M$ is the sum of the first block and kernel of $A$, so its generic dimension is $d_2+h_1(\rep(\underline{d},\mathcal{A}))$ as needed.
\end{proof}

\begin{lemma}
    With notation as above

    \[d_1'=\max(n\cdot d_2+d_0-(n^2-1)\cdot d_2,n\cdot d_2)\]
\end{lemma}

\begin{proof}
    Since $\socdim$ is upper semicontinuous and $\rep(\underline{d},\mathcal{A})$ irreducible, there exists some representation $M\in\rep(\underline{d},\mathcal{A})$ with minimal value $\socdim M$. In particular $\socdim_1 M=d_2+(d_1-n\cdot d_2)$. Up to isomorphism, we may assume $M$ has the standard block structure, with 

    \[A=\begin{pmatrix}
        0_k & I & & 0\\
        \vdots & & \ddots & \\
        0_K & 0 & & I
    \end{pmatrix}\]

    where $I=(\lambda_1,\dots,\lambda_n)$, since this is the only tuple with such a dimension vector and minimal kernel.

    Now $\soc^2(M)$ contains exactly all vectors that map into the space $k^{d_0'}\oplus 0$. Hence it is the direct sum of $\rad M$ and the common kernel of the maps

    \[C_i': M/\rad M\xrightarrow{C_i}\rad M/\rad^2 M\rightarrow \rad M/\soc M\]

    Where the matrices $C_i'$ are exactly the matrices we get if we erase the first $d_1-nd_2$ rows from matrices $C_i$.
    This means it is enough to determine the common kernel of matrices $C_i'$.

    First notice that for a given $A$ the relation

    \[S=\sum A_iC_i=0\]

    is equivalent to the relation

    \[S'=\sum A_i'C_i'=0\]

    where $A_i'$ is the matrix we get from $A_i$ if we erase the first $d_1-nd_2$ columns. That means that independently of the first $d_1-nd_2$ columns of $C_i$, the set of tuples $(C_i')_i$ for which $M$ is a representation of $\mathcal{A}$ is exactly the set of tuples $(C_i')_i$ that satisfy relation $S'$.

    The minimal dimension of the common kernel is thus

    \[d_0-\dim V\]

    where $V$ is the dimension of the space of solutions of the equation 
    
    \[\sum A_i'v_i=0\]

    By lemma~\ref{relation}

    \[\dim V=\max(0,(n^2-1)\cdot d_2)\]

    since $\rep(\underline{d},\mathcal{A})$ is nonempty, $\dim V=0$ is not possible, so we have

    \[\dim \left(\bigcap \ker C_i'\right)=\max (0,d_0-(n^2-1)\cdot d_2)\]

    The result follows.
\end{proof}

This gives the generic socle layering for $\rep(\underline{d},\mathcal{A})$;

\begin{cor}\label{generic_socdim}
    Let $\underline{d}=(d_0,d_1,d_2)$ be a dimension vector of radical layering, such that $\rep(\underline{d},\mathcal{A})$ is nonempty. Further suppose $h_1(\rep(\underline{d},\mathcal{A}))=h=d_1-nd_2>0$. Then the generic value of the dimension of the socle layering is

    \[\socdim \underline{d}=\begin{cases}
        (d_2+h,d_1-h,d_0) & (n^2-1)\cdot d_2\geq d_0,\\
        (d_2+h,d_1-h+d_0-(n^2-1)\cdot d_2, (n^2-1)d_2) & \text{oterwise.}
    \end{cases}\]
\end{cor}

Dually we can compute the generic radical layering for $\rep^*(\underline{d},\mathcal{A})$

\begin{cor}
    Let $\underline{d}=(d_0,d_1,d_2)$ be a dimension vector of radical layering, such that $\rep^*(\underline{d},\mathcal{A})$ is nonempty. Further suppose $h'_1(\rep^*(\underline{d},\mathcal{A}))=h=d_1-nd_2>0$. Then the generic value of the dimension of radical layering is

    \[\raddim \underline{d}=\begin{cases}
        (d_2+h,d_1-h,d_0) & (n^2-1)\cdot d_2\geq d_0,\\
        (d_2+h,d_1-h+d_0-(n^2-1)\cdot d_2, (n^2-1)d_2) & \text{oterwise.}
    \end{cases}\]
\end{cor}

We can use this and lemma~\ref{lemma_min_soc_rad} to determine the possible irreducible components 

\begin{lemma}
    Suppose $\underline{d}$ is a dimension vector, such that $\rep(\underline{d},\mathcal{A})$ is nonempty and 

    \[\raddim(\socdim(\underline{d}))=\underline{d}\]

    then $\socdim(\underline{d})=(d_2,d_1,d_0)$ or $\underline{d}$ is one of the exceptional dimension vectors

    \[\layer\in \Big\{(a,n(a-1),(n^2-1)(a-1)),((n^2-1)(a-1),n(a-1)+1,a-1) \Big\}\]
\end{lemma}

\begin{proof}
    Suppose this is not true. Then

    \[h_0(\rep(\underline{d},\mathcal{A}))=0 \qquad \text{and} \qquad h_1(\rep(\underline{d},\mathcal{A}))>0\]

    Hence 

    \[\underline{d'}=\socdim \underline{d}=\begin{cases}
        (d_2+h,d_1-h,d_0) & (n^2-1)\cdot d_2\geq d_0,\\
        (d_2+h,d_1-h+d_0-(n^2-1)\cdot d_2, (n^2-1)d_2) & \text{otherwise.}
    \end{cases}\]

    First suppose $\underline{d}'=(d_2+h,d_1-h,d_0)$. Using the same argument as for radical layering, we can show $h_0'=0$, or $\underline{d}'$ is an exceptional socle layering. In the first case

    \[\raddim(\underline{d}')=\begin{cases}
        (d_0,d_1-h,d_2+h)\\
        (d_0+(d_1-h-h'),\cdot,\cdot)
    \end{cases}\]

    depending on whether $h'=h_1(\rep^*(\underline{d}',\mathcal{A}))=0$ or not. Neither of these equals $\underline{d}$ as needed. 
    
    In the second case, $\underline{d}$ must be exceptional as well, otherwise $\raddim(\underline{d}')\neq \underline{d}$. 
    
    Hence $\underline{d}'=(d_2+h,d_1-h,d_0)$ is not possible. We are left with the case $\underline{d}'=(d_2+h,d_1-h+d_0-(n^2-1)\cdot d_2, (n^2-1)d_2)$. Again we need $h'_0(\rep^*(\underline{d},\mathcal{A}))=0$ and hence $h'=h'_1(\rep^*(\underline{d},\mathcal{A}))\neq 0$, so the possible values for $\raddim(\underline{d}')$ are

    \[\raddim(\underline{d}')=\begin{cases}
        (\cdot,\cdot, d_2+h)\\
        (\cdot,\cdot, (n^2-1)^2d_2)
    \end{cases}\]

    neither of which can be $\underline{d}$ again.
\end{proof}

The dimension vectors that satisfy the conclusion of the previous lemma are all minimal in the image $\vartheta(\rep_d(\mathcal{A}))$, we have shown this for exceptional vectors earlier and for the others it follows from lemma~\ref{socdim_raddim_complement}. Hence these are all irreducible components. With this, we have proven theorem~\ref{thm_components}. Theorem~\ref{thm_gen_rad} is then the same as lemma~\ref{generic_socdim} applied for the irreducible components we found.

\subsection*{Acknowledgement}

The results of this article were produced as part of my master's thesis at the University of Ljubljana.
I thank my master's thesis mentor K. Šivic for his mentorship.
This work was partially supported by the grant G0B3123N from FWO.

\bibliography{vir}

\Addresses

\end{document}